\documentclass[12pt]{article}
\usepackage{amsmath,mathdots}
\usepackage{amssymb}
\usepackage{latexsym}
\usepackage{amsthm}
\usepackage{amsmath,amssymb,amsfonts,amsthm}
\usepackage{mdframed}
\usepackage{graphicx}
\usepackage{stmaryrd}
\usepackage{gensymb}

\usepackage{stmaryrd}
\usepackage{multirow}

\usepackage{graphicx}
\usepackage{tikz}
\usetikzlibrary{matrix,arrows}
\usetikzlibrary{positioning}
\usetikzlibrary{fit}
\usetikzlibrary{patterns}

\usepackage[colorlinks,
linkcolor=blue,
anchorcolor=blue,
citecolor=blue
]{hyperref}

\newtheorem{thm}{Theorem}[section]

\newtheorem{lem}[thm]{Lemma}

\usepackage{graphicx}
\usepackage{tikz}
\usetikzlibrary{matrix,arrows}
\usetikzlibrary{positioning}
\usetikzlibrary{fit}
\usetikzlibrary{patterns}

\newcommand{\dbrac}[1]{{\llbracket #1 \rrbracket}} 	
\newcommand{\boks}[2]{({#1, #2})}   

\newcommand{\pattern}[4]{										
	\raisebox{0.6ex}{
		\begin{tikzpicture}[scale=0.35, baseline=(current bounding box.center), #1]
		\foreach \x/\y in {#4}		\fill[gray!20] (\x,\y) rectangle +(1,1);
		\draw (0.01,0.01) grid (#2+0.99,#2+0.99);
		\foreach \x/\y in {#3}		\filldraw (\x,\y) circle (6pt);
		\end{tikzpicture}}
}

\DeclareMathOperator{\red}{red}

\begin{document}

\begin{center}
{\large \bf  Distributions of mesh patterns of short lengths on king permutations}
\end{center}

\begin{center}
Dan Li$^{a}$ and Philip B. Zhang$^{b*}$ 
\\[6pt]

$^{a,b}$School of Mathematical Sciences  \& Institute of Mathematics and Interdisciplinary Sciences\\
Tianjin Normal University, Tianjin, China\\[6pt]

Email:  
          $^{a}${\tt lidan\_ld@yeah.net} \ 
           $^{b*}${\tt zhang@tjnu.edu.cn}
\end{center}

\noindent\textbf{Abstract.}
Br\"{a}nd\'{e}n and Claesson introduced the concept of mesh patterns in 2011, and since then, these patterns have attracted significant attention in the literature. Subsequently, in 2015, Hilmarsson \emph{et al.} initiated the first systematic study of avoidance of mesh patterns, while Kitaev and Zhang conducted the first systematic study of the distribution of mesh patterns in 2019.
A permutation $\sigma = \sigma_1 \sigma_2 \cdots \sigma_n$ in the symmetric group $S_n$ is called a king permutation if $\left| \sigma_{i+1}-\sigma_i \right| > 1$ for each $1 \leq i \leq n-1$. Riordan derived a recurrence relation for the number of such permutations in 1965. The generating function for king permutations was obtained by Flajolet and Sedgewick in 2009.  In this paper, we initiate a systematic study of the distribution of mesh patterns on king permutations by finding distributions for 22 mesh patterns of short length. \\



\section{Introduction}\label{intro}

The field of permutation patterns has garnered significant interest in the literature (see~\cite{Kit} and references therein). This area of research continues to increase gradually. 
The notion of a {\em mesh pattern}, which generalizes several classes of patterns, was introduced by Br\"and\'en and Claesson \cite{BrCl} to provide explicit expansions for certain permutation statistics as, possibly infinite, linear combinations of (classical) permutation patterns. 

A tuple $(\tau,R)$, where $\tau$ signifies a permutation comprising $k$ elements and $R$ is a subset of $\dbrac{0,k} \times \dbrac{0,k}$ (with
$\dbrac{0,k}$ representing the integer set $\{ 0, 1 \ldots  k \}$), constitutes a
\emph{mesh pattern} of length $k$.
We define $\boks{i}{j}$ as the rectangular area whose vertices are positioned at $(i,j), (i,j+1),
(i+1,j+1)$, and $(i+1,j)$. In this context, horizontal lines symbolize the values,  and vertical lines denote the positions in the pattern. Mesh patterns can visually be represented by shading the boxes in $R$. For instance, the mesh pattern with $\tau=123$ and $R = \{\boks{1}{1},\boks{2}{2}\}$ is illustrated as follows.
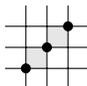
\[
\pattern{scale=0.8}{3}{1/1,2/2,3/3}{1/1,2/2}
\]

Numerous papers have been dedicated to the study of mesh patterns and their generalizations; see, for example, \cite{AKV,BG,Borie,HZ,JKR,KL,KR1,T1,T2,TR}. Nevertheless, the first systematic study of mesh patterns was not conducted until \cite{Hilmarsson2015Wilf}, where 25 out of 65 non-equivalent cases of patterns of length 2 were solved in the context of {\em avoidance}. In \cite{KZ}, a systematic study of {\em distributions} of mesh patterns was initiated by providing 22 distribution results  for the patterns considered in \cite{Hilmarsson2015Wilf}, which includes 14 distributions where avoidance was previously unknown. Kitaev, Zhang, and Zhang~\cite{KZZ} present far-reaching generalizations of 8 known distribution results and 5 known avoidance results concerning mesh patterns. They achieve this by providing distribution or avoidance formulas for certain infinite families of mesh patterns in terms of formulas for smaller patterns. Furthermore, as a corollary to a general result, in \cite{KZZ}  they determine the distribution of one additional mesh pattern of length 2.

In this paper, we extend the distribution results on mesh patterns by Kitaev and Zhang \cite{KZ} to the cases of king permutations that have been explored in~\cite{BESM-0,BESM,Claesson,FS,Riordan} (see Section~\ref{prelim} for definitions). We initiate a systematic study of the distribution of mesh patterns on king permutations and derive the distributions for 22 specific patterns.
Among the short mesh patterns considered by Kitaev and Zhang \cite{KZ} for all permutations, we derive distributions for 22 patterns on king permutations. 
Table~\ref{tab-1} provides an overview of our enumerative results. We note that, unlike the case of all permutations  \cite{KZ}, all distributions we obtain are different except for the patterns  $\pattern{scale = 0.5}{2}{1/1,2/2}{0/0,0/1,0/2,1/0,1/1,1/2,2/0,2/1,2/2 }$, $\pattern{scale = 0.5}{2}{1/1,2/2}{0/1,1/1,1/2,1/0,1/2,2/1}$, $\pattern{scale=0.5}{2}{1/1,2/2}{0/1,1/2,2/0,1/0,1/1,2/1,0/2}$, $\pattern{scale=0.5}{2}{1/1,2/2}{0/1,1/2,0/0,2/2,1/0,1/1,2/1}$, $\pattern{scale=0.5}{2}{1/1,2/2}{0/1,1/2,0/0,1/0,1/1,2/1}$, and $\pattern{scale=0.5}{2}{1/1,2/2}{0/1,1/2,1/0,1/1,2/1,0/2}$ (discussed in Theorem~\ref{triv-thm-11}) that cannot occur in king permutations.

The paper is organized as follows. In Section~\ref{prelim}, we provide necessary definitions and known results, and derive several distributions to be used throughout the paper. In Section~\ref{trivial-sec}, we consider mesh patterns that either can occur at most once or whose occurrences can be easily found. Other distributions are derived in Section~\ref{other-distr-sec}. Finally, in Section~\ref{final-sec}, we provide concluding remarks and list the short mesh patterns from \cite{KZ,KZZ} for which we were unable to find distributions.

  \begin{table}[!ht]
 	{
 		\renewcommand{\arraystretch}{1.3}
\begin{footnotesize}
 \begin{center} 
 		\begin{tabular}{|c|c|c||c|c|c|}
 			\hline
 			{\footnotesize Nr.\ } & {\footnotesize Repr.\ $p$}  & {\footnotesize Distribution}  &  {\footnotesize Nr.\ } & {\footnotesize Repr.\ $p$}  & {\footnotesize Distribution}  
 			\\[5pt]
 			\hline		\hline
 			 $X$ &  \pattern{scale=0.5}{1}{1/1}{0/1,1/0} & Theorem~\ref{thm-length-1}
 & 
  		22 & $\pattern{scale = 0.5}{2}{1/1,2/2}{0/1,1/2,0/0,2/0,2/2,1/1}$ & Theorem~\ref{triv-thm-22}
\\[5pt]
 			\hline
 		$X'$	 & \pattern{scale=0.5}{1}{1/1}{0/0,1/1} & Theorem~\ref{thm-length-1}
 			&
		27 & $\pattern{scale=0.5}{2}{1/1,2/2}{0/1,2/0,2/2,1/0,1/1,0/2}$    & Theorem~\ref{thm-pat-27}
 			\\[5pt] 
 			\hline
 		     10 &  $\pattern{scale = 0.5}{2}{1/1,2/2}{0/0,0/1,0/2,2/0,2/1,2/2}$ &   Theorem~\ref{triv-thm-10}  & 
 			 	28 & $\pattern{scale=0.5}{2}{1/1,2/2}{0/1,1/2,0/0,2/2,1/0,2/1}$    &  Theorem~\ref{thm-pat-28} 
 		    \\[5pt] 
 			\hline
 			11 & $\pattern{scale = 0.5}{2}{1/1,2/2}{0/0,0/1,0/2,1/0,1/1,1/2,2/0,2/1,2/2 }$ & Theorem~\ref{triv-thm-11}
 			&
			30 & $\pattern{scale=0.5}{2}{1/1,2/2}{0/1,1/2,2/0,1/0,1/1,2/1,0/2}$ & Theorem~\ref{triv-thm-11}  
 			\\[5pt]
 			\hline
 			 12 & $\pattern{scale = 0.5}{2}{1/1,2/2}{0/0,0/1,0/2,1/0,2/0}$ & Theorem~\ref{triv-thm-12}
 			&
			33 & $\pattern{scale=0.5}{2}{1/1,2/2}{0/1,1/2,2/0,1/0,0/2,2/1}$   & Theorem~\ref{thm-pat-33} 
 			\\[5pt]
 			\hline
 			 13 & $\pattern{scale = 0.5}{2}{1/1,2/2}{0/0,0/1,0/2,1/0,1/2,2/0,2/1,2/2}$ & Theorem~\ref{triv-thm-13}
 			&
 			34 & $\pattern{scale=0.5}{2}{1/1,2/2}{0/1,1/2,0/0,2/2,1/0,1/1,2/1}$ &  Theorem~\ref{triv-thm-11} 
  			\\[5pt]
 			\hline
 			14 &   $\pattern{scale = 0.5}{2}{1/1,2/2}{0/1,1/1,1/2,1/0,1/2,2/1}$ & Theorem~\ref{triv-thm-11}
 			& 
			36 & $\pattern{scale=0.5}{2}{1/1,2/2}{0/1,1/2,0/0,1/0,1/1,2/1}$   &   Theorem~\ref{triv-thm-11}  
 \\[5pt]
 			\hline
 			 16 &  $\pattern{scale = 0.5}{2}{1/1,2/2}{0/1,2/0,1/0,0/2}$ &  Theorem~\ref{thm-pat-16}
 			&
			45 & $\pattern{scale=0.5}{2}{1/1,2/2}{0/1,1/2,1/0,1/1,2/1,0/2}$     &  Theorem~\ref{triv-thm-11} 
\\[5pt]
 			\hline
 			17 & $\pattern{scale = 0.5}{2}{1/1,2/2}{0/1,1/2,0/0,2/0,1/0,0/2,2/1}$ & Theorem~\ref{thm-pat-17}
 			&
 			55 & $\pattern{scale = 0.5}{2}{1/1,2/2}{0/1,1/2,0/0,2/0,1/1,2/1}$     &  Theorem~\ref{thm-pat-55} 
 \\[5pt]
 			\hline
 			19 & $\pattern{scale = 0.5}{2}{1/1,2/2}{0/1,0/2,1/1,1/2,2/0,2/2}$ &  Theorem~\ref{triv-thm-19}
 			&
  			 63 & $\pattern{scale = 0.5}{2}{1/1,2/2}{0/1,1/2,0/0,2/1,2/0}$ & Theorem~\ref{thm-pat-63}
\\[5pt]
 			\hline
            20 & $\pattern{scale = 0.5}{2}{1/1,2/2}{0/0,0/1,0/2,1/1,1/2,2/0,2/1}$ &  Theorem~\ref{triv-thm-20}  
 			&
  			64 & $\pattern{scale = 0.5}{2}{1/1,2/2}{0/1,1/2,2/0,0/2,1/1}$ & Theorem~\ref{thm-pat-64} 			
\\[5pt]
 			\hline 
 		\end{tabular}
\end{center} 	
\end{footnotesize}}
 	\caption{The patterns for which we derive distributions in this paper. Patterns' numbers come from \cite{Hilmarsson2015Wilf,KZ,KZZ}.}\label{tab-1}
\end{table}

\section{Preliminaries}\label{prelim}
 
Let $S_n$ be the set of all permutations of length $n$, which we call $n$-permutations. For example, $S_3=\{123, 132, 213, 231, 312, 321\}$. For $\sigma = \sigma_1 \sigma_2 \cdots \sigma_n \in S_n$, let $\sigma^r=\sigma_n\sigma_{n-1}\cdots\sigma_1$ and $\sigma^c= (n + 1 - \sigma_1)(n + 1 - \sigma_2)\cdots (n + 1-\sigma_n)$
denote the {\em reverse} and {\em complement} of $\sigma$, respectively. Then $\sigma^{rc}=(n + 1-\sigma_n)(n + 1-\sigma_{n-1})\cdots (n + 1-\sigma_1)$. For example, $(2431)^r=1342$, $(2431)^c=3124$ and $(2431)^{rc}=4213$. If $\sigma$ is any sequence of distinct numbers, then the {\em reduced form} of $\sigma$, denoted as $\red(\sigma)$, is obtained from $\sigma$ by replacing the smallest element by 1, the next smallest element by 2, etc. For example, $\red(2648)=1324$. 
 
A permutation $\sigma = \sigma_1 \sigma_2 \cdots \sigma_n \in S_n$ is called a {\em king permutation} if $\left| \sigma_{i+1}-\sigma_i \right| > 1$ for each $1 \leq i \leq n-1$.
Let $K_n$ be the set of all king permutations of length $n$, which we call  king $n$-permutations, and $K=\cup_{n\geq 0}K_n$. For example, $K_4 = \{3142, 2413\}$. For a pattern $p$ and a king permutation $\sigma$, we let $p(\sigma)$ denote the number of occurrences of $p$ in $\sigma$. Also, let $K_n(p)$ denote the set of all king $n$-permutations avoiding $p$ and $K(p)=\cup_{n\geq 0}K_n(p)$. King permutations are precisely  those permutations that avoid the patterns $\pattern{scale = 0.5}{2}{1/1,2/2}{0/1,1/1,1/2,1/0,1/2,2/1}$ and $\pattern{scale = 0.5}{2}{2/1,1/2}{0/1,1/1,1/2,1/0,1/2,2/1}$. Let $K_{n}^{s}$ be the set of all king $n$-permutations that do not begin with the smallest element, and let $K_{n}^{\ell}$ be the set of all king $n$-permutations that do not end with the largest element. Also, $K^s=\cup_{n\geq 0}K^s_n$ and $K^{\ell}=\cup_{n\geq 0}K^{\ell}_n$. We also let $K^{s\ell}=K^{s}\cap K^{\ell}$ and $K^{\ell s}$ is obtained from $K^{s\ell}$ by applying the complement operation to each permutation.
 
Let $A_n$ be the number of all king $n$-permutations.  Riordan \cite{Riordan} derived a recurrence relation for $A_n$ in 1965: $A_0 =A_1 =1$, $A_2 =A_3 =0$, and for $n\geq 4$,
$$A_n = (n+1)A_{n-1} -(n-2)A_{n-2} -(n-5)A_{n-3} +(n-3)A_{n-4}.$$ The initial values for $A_n$ are
$$1, 1, 0, 0, 2, 14, 90, 646, 5242, 47622, 479306, 5296790, 63779034,\cdots.$$
It is known \cite{Claesson} that for $n\geq 4$, 
$$A_n=n!+\sum_{k=1}^{n}(-1)^k\sum_{i=1}^{k}\binom{k-1}{i-1}\binom{n-k}{i}2^i(n-k)!.$$
Moreover, Flajolet and Sedgewick \cite{FS} showed that
\begin{equation}\label{gf-kings}A(t)=\sum_{n\geq 0}A_n t^n=\sum_{n\geq 0}\frac{n!t^n (1-t)^n}{(1+t)^n},\end{equation} 
which will be used frequently in our paper. Let $B_n$ (resp., $C_n$) be the number of all king $n$-permutations that do not begin with the smallest element (resp., and do not end with the largest element). In particular, $B_0=C_0=1$ and $B_1=C_1=0$. By reversing and complementing king permutations, $B_n$ is also equal to the number of all king $n$-permutations that do not end with the largest element. Let $B(t)=\sum_{n\geq 0}B_n t^n$ and $C(t)=\sum_{n\geq 0}C_n t^n$.
 
\begin{lem}\label{lem-length-1} We have
\begin{equation}\label{B(x)}
B(t)=\frac{A(t)}{1+t}=\sum_{n\geq 0}\frac{n!t^n (1-t)^n}{(1+t)^{n+1}},
\end{equation}
\begin{equation}\label{C(x)}
C(t)=\frac{t}{1+t}+\frac{A(t)}{(1+t)^2}=\frac{t}{1+t}+\sum_{n\geq 0}\frac{n!t^n (1-t)^n}{(1+t)^{n+2}}.
\end{equation}
The initial terms in the expansion of $B(t)$ and $C(t)$ are
\begin{footnotesize}
$$B(t)=1+2t^4+12t^5+78t^6+568t^7+4674t^8+42948t^9+436358t^{10}+\cdots;$$
$$C(t)=1+2t^4+10t^5+68t^6+500t^7+4174t^8+38774t^9+397584t^{10}+\cdots.$$
\end{footnotesize}
\end{lem}
 
\begin{proof} We claim that 
\begin{equation}\label{av-length-1b}B(t)+tB(t)=A(t).\end{equation}
Indeed, each $\sigma\in K$, counted by the $A(t)$ on the right-hand side of \eqref{av-length-1b}, either belongs to $K^s$ (and thus is counted by the $B(t)$ on the left-hand side of \eqref{av-length-1b}), or begins with the smallest element (and thus is counted by the $tB(t)$). The formula for $B(t)$ now follows from (\ref{gf-kings}) and (\ref{av-length-1b}).

Next, we claim that 
\begin{equation}\label{rec-for-C}C(t)=A(t)-t-2t\big(B(t)-1\big)+t^2\big(C(t)-1\big).\end{equation}
Indeed, any king permutation counted by $C(t)$ can be obtained from all king permutations by subtracting the permutation 1 (the term $t$) and any other king permutation of length at least 2 that starts with 1 (the term $t\big(B(t)-1\big)$) or ends with the largest element (the term $t\big(B(t)-1\big)$). However, king permutations beginning with 1 and ending with the largest element (the term $t^2\big(C(t)-1\big)$) have been subtracted twice, and they need to be added back.
\end{proof}
 
Occurrences of the pattern $\pattern{scale=0.5}{1}{1/1}{0/1,1/0}$ in permutations are known as {\em strong fixed points}. The distribution of strong fixed points on permutations is derived in~\cite{KZ}. In the next theorem we derive distributions of strong fixed points on (restricted) king permutations. These distributions will be frequently used throughout the paper.

\begin{thm}\label{thm-length-1} 
Let
\begin{footnotesize}
$$A(t,u):=\sum_{n\geq 0}t^n\sum_{\sigma\in K_n}u^{\pattern{scale=0.5}{1}{1/1}{0/1,1/0}(\sigma)},\ B(t,u):=\sum_{n\geq 0}t^n\sum_{\sigma\in K_{n}^{s}}u^{\pattern{scale=0.5}{1}{1/1}{0/1,1/0}(\sigma)},\ C(t,u):=\sum_{n\geq 0}t^n\sum_{\sigma\in K_{n}^{s\ell}}u^{\pattern{scale=0.5}{1}{1/1}{0/1,1/0}(\sigma)},$$
\end{footnotesize}
and let $P(t)=A(t,0)$ be the generating function for $K(\hspace{-1mm}\pattern{scale=0.5}{1}{1/1}{1/0,0/1})$.
Then, $P(t)$ is also the generating function for $K(\hspace{-1mm}\pattern{scale=0.5}{1}{1/1}{0/0,1/1})$. Moreover, 
\begin{equation}\label{4-equations}
K(\hspace{-1mm}\pattern{scale=0.5}{1}{1/1}{1/0,0/1})=K^s(\hspace{-1mm}\pattern{scale=0.5}{1}{1/1}{1/0,0/1})=K^{\ell}(\hspace{-1mm}\pattern{scale=0.5}{1}{1/1}{1/0,0/1})=  K^{s\ell}(\hspace{-1mm}\pattern{scale=0.5}{1}{1/1}{1/0,0/1})=K^{\ell s}(\hspace{-1mm}\pattern{scale=0.5}{1}{1/1}{1/0,0/1}),
\end{equation}
\begin{footnotesize}
\begin{equation}\label{avd-begin-smallest-dist}P(t)=\frac{(1+t)A(t)}{1+t+tA(t)},\end{equation}
\begin{equation}\label{fixed-strict-dist}A(t,u)=\sum_{n\geq 0}t^n\sum_{\sigma\in K_n}u^{\pattern{scale=0.5}{1}{1/1}{0/0,1/1}(\sigma)}=\frac{(1+ut)(1+t)A(t)}{1+t(1+u+ut)+t(1-u)A(t)},\end{equation}
\begin{equation}\label{not-begin-end-extreme}
B(t,u)=\sum_{n\geq 0}t^n\sum_{\sigma\in K_{n}^{\ell}}u^{\pattern{scale=0.5}{1}{1/1}{1/0,0/1}(\sigma)}=\frac{(1+t)A(t)}{1+t(1+u+ut)+t(1-u)A(t)},
\end{equation}
\begin{equation}\label{not-begin-end-beginning}
C(t,u)=\sum_{n\geq 0}t^n\sum_{\sigma\in K^{\ell s}_n}u^{\pattern{scale=0.5}{1}{1/1}{0/0,1/1}(\sigma)}=\frac{ut}{1+ut}+\frac{(1+t)A(t)}{(1+ut)\big(1+t(1+u+ut)+t(1-u)A(t)\big)},
\end{equation}
\end{footnotesize}
where $A(t)$ is given by (\ref{gf-kings}). The initial terms in the expansion of  $A(t,u)$, $B(t,u)$, and $C(t,u)$  are
\begin{footnotesize}
$$A(t,u)=1+ut+2t^4+(10+4u)t^5+(68+20u+2u^2)t^6+(500+136u+10u^2)t^7+\cdots;$$
$$B(t,u)=1+2t^4+(10+2u)t^5+(68+10u)t^6+(500+68u)t^7+(4174+500u)t^8+\cdots;$$
$$C(t,u)=1+2t^4+10t^5+68t^6+500t^7+4174t^8+\cdots.$$
\end{footnotesize} 
\end{thm}

\begin{proof} 
Because reverse and/or complement of a king permutation is a king permutation, $P(t)$ is also the generating function for $K(\hspace{-1mm}\pattern{scale=0.5}{1}{1/1}{0/0,1/1})$, and for $A(t,u)$, we only need to consider the pattern $\pattern{scale=0.5}{1}{1/1}{0/1,1/0}$. 
Moreover, under the composition of reverse and complement operations, permutations in $K^s_n$ go bijectively to permutations in $K^{\ell}_n$, so for $B(t,u)$ it is sufficient to consider $K^s_n$.

Note that (\ref{4-equations}) follows from $P(t)=A(t,0)=B(t,0)=C(t,0)$ as no king permutation avoiding  $\pattern{scale=0.5}{1}{1/1}{0/1,1/0}$ can begin with 1 or end with the largest element. 

We claim that 
\begin{equation}\label{av-length-1}P(t)+tP(t)B(t)=A(t).\end{equation}
Indeed, each $\sigma\in K$, counted by the $A(t)$ on the right-hand side of \eqref{av-length-1}, either avoids $\pattern{scale=0.5}{1}{1/1}{0/1,1/0}$ (and thus is counted by the $P(t)$ on the left-hand side of \eqref{av-length-1}), or contains at least one occurrence of $\pattern{scale=0.5}{1}{1/1}{0/1,1/0}$. Consider the leftmost occurrence of $\pattern{scale=0.5}{1}{1/1}{0/1,1/0}$, an element $a$ in $\sigma$, which corresponds to the $t$ in $tP(t)B(t)$. To the left of $a$, we must have a $\pattern{scale=0.5}{1}{1/1}{0/1,1/0}$ -avoiding permutation $\sigma'\in K^{\ell}$ ($\sigma'$ cannot end with its largest element) formed by the smallest elements of $\sigma$ and counted by $P(t)$ because any king permutation avoiding $\pattern{scale=0.5}{1}{1/1}{0/1,1/0}$ cannot end with the largest element. To the right of $a$, we can have any permutation  $\sigma''$, $\red(\sigma'')\in K^s$,  formed by the largest elements of $\sigma$ and counted by $B(t)$. 
Hence, we obtain the equation \eqref{av-length-1} due to the independence of the choices of $\sigma'$ and $\sigma''$. 

Similarly, we can derive 
\begin{equation}\label{length-1-formula}P(t)+utP(t)B(t,u)=A(t,u)\end{equation}
where the element $a$ will contribute the factor $ut$. 

Our derivation of the following equation is similar to that of (\ref{av-length-1b}).
\begin{equation}\label{length-1-formulab}B(t,u)+utB(t,u)=A(t,u).\end{equation}  

Combining \eqref{av-length-1}--\eqref{length-1-formulab}, we obtain the desired result.

Next, we claim that 
\begin{equation}\label{rec-for-C-u}C(t,u)=A(t,u)-ut-2ut\big(B(t,u)-1\big)+(ut)^2\big(C(t,u)-1\big).\end{equation}
Indeed, any king permutation counted by $C(t,u)$ can be obtained from all king permutations by subtracting the permutation 1 (the term $ut$) and any other king permutation of length at least 2 that start with 1 (the term $ut\big(B(t,u)-1\big)$) or end with the largest element (the term $ut\big(B(t,u)-1\big)$). However, king permutations beginning with 1 and ending with the largest element (the term $u^2t^2\big(C(t,u)-1\big)$) have been subtracted twice.
\end{proof}

\section{``Trivial'' distributions}\label{trivial-sec}

A ``trivial'' distribution refers to a situation where either the pattern in question can occur at most once, or the occurrences of the pattern can be easily found. 

\begin{thm}\label{triv-thm-11}  For  patterns $\pattern{scale = 0.5}{2}{1/1,2/2}{0/0,0/1,0/2,1/0,1/1,1/2,2/0,2/1,2/2 }$, $\pattern{scale = 0.5}{2}{1/1,2/2}{0/1,1/1,1/2,1/0,1/2,2/1}$, $\pattern{scale=0.5}{2}{1/1,2/2}{0/1,1/2,2/0,1/0,1/1,2/1,0/2}$, $\pattern{scale=0.5}{2}{1/1,2/2}{0/1,1/2,0/0,2/2,1/0,1/1,2/1}$, $\pattern{scale=0.5}{2}{1/1,2/2}{0/1,1/2,0/0,1/0,1/1,2/1}$ and $\pattern{scale=0.5}{2}{1/1,2/2}{0/1,1/2,1/0,1/1,2/1,0/2}$, $A(t,u)=A(t)$ given by (\ref{gf-kings}).
\end{thm}

\begin{proof} 
By definition, no king permutation can contain occurrences of any of the six patterns, which completes the proof. \end{proof}

\begin{thm}\label{triv-thm-10} 
For $n\geq 2$, there are $A_n/2$ king $n$-permutations avoiding the pattern $\pattern{scale = 0.5}{2}{1/1,2/2}{0/0,0/1,0/2,2/0,2/1,2/2}$, and $A_n/2$ king $n$-permutations containing it exactly once. 
\end{thm}

\begin{proof} 
The only possible occurrence of $\pattern{scale = 0.5}{2}{1/1,2/2}{0/0,0/1,0/2,2/0,2/1,2/2}$ must involve the leftmost and rightmost elements. Clearly, reverse of a king permutation is a king permutation. Hence, exactly in half of king  $n$-permutations the leftmost element is smaller than the rightmost element. 
\end{proof}

\begin{thm}\label{triv-thm-12} 
Let $p=\pattern{scale = 0.5}{2}{1/1,2/2}{0/0,0/1,0/2,1/0,2/0}$, $E(t,u)=\sum_{n\geq 0}t^n\sum_{\sigma\in K_n}u^{p(\sigma)}$, 
and $P(t)$ be the generating function for $K(p)$. Then, $$P(t)=t+\frac{A(t)}{1+t},\ \ \ E(t,u)=\frac{A(t)}{1+t}+\frac{tA(ut)}{1+ut}.$$ where $A(t)$ is given by (\ref{gf-kings}). The initial terms in the expansion of  $E(t,u)$ are
\begin{footnotesize}
$$1+t+2t^4+(12+2u^4)t^5+(78+12u^5)t^6+(568+78u^6)t^7+(4674+568u^7)t^8+\cdots.$$
\end{footnotesize}
\end{thm}

\begin{proof}
Any occurrence of $p$ must start with the element 1 placed in position~1. Suppose a permutation $\sigma\in K$ begins with 1. To the right of 1 in $\sigma$,  we can have any non-empty permutation  $\sigma''$, $\red(\sigma'')\in K^s$,  formed by the largest elements of $\sigma$ and counted by $B(t)-1$. The permutations not beginning with 1 are counted by $P(t)$. To get avoidance, we take all king permutations and subtract those that contain the pattern $p$:
$$P(t)=A(t)-t\big(B(t)-1\big)$$ 
where the factor of $t$ corresponds to 1 in $\sigma$. By Lemma~\ref{lem-length-1}, we obtain the formula for $P(t)$.

Similarly, we can get
$$E(t,u)=P(t)+t\big(B(ut)-1\big)$$
where we used the fact that if a permutation begins with 1, then any element to the right of 1 will contribute an occurrence of  $\pattern{scale = 0.5}{2}{1/1,2/2}{0/0,0/1,0/2,1/0,2/0}$.
\end{proof}

\begin{thm}\label{triv-thm-13} 
Let $p = \pattern{scale = 0.5}{2}{1/1,2/2}{0/0,0/1,0/2,1/0,1/2,2/0,2/1,2/2}$, $E(t,u)=\sum_{n\geq 0}t^n\sum_{\sigma\in K_n}u^{p(\sigma)}$, 
and $P(t)$ be the generating function for $K(p)$. Then, $$P(t)=\frac{t^2}{1+t}+\frac{(1+2t)A(t)}{(1+t)^2},\ \ \ E(t,u)=\frac{t^2(1-u)}{1+t}+\frac{(1+2t+ut^2)A(t)}{(1+t)^2},$$
where $A(t)$ is given by (\ref{gf-kings}). The initial terms in the expansion of  $E(t,u)$ are
\begin{footnotesize}
$$1+t+2t^4+14t^5+(88+2u)t^6+(636+10u)t^7+(5174+68u)t^8+\cdots.$$
\end{footnotesize}
\end{thm}

\begin{figure}
\begin{center}
	\begin{tikzpicture}[scale=0.6, baseline=(current bounding box.center)]
	\foreach \x/\y in {0/0,0/1,0/2,1/0,1/2,2/0,2/1,2/2}		    
		\fill[gray!20] (\x,\y) rectangle +(1,1);
	\draw (0.01,0.01) grid (2+0.99,2+0.99);
	\filldraw (1,1) circle (3pt) node[above left] {{\small $a$}};
	\filldraw (2,2) circle (3pt) node[above right] {{\small $b$}};
	\node  at (1.5,1.5) {$A$};
	\end{tikzpicture}
\end{center}
\caption{Related to the proof of Theorem~\ref{triv-thm-13}}\label{pic-triv-thm-13}
\end{figure}
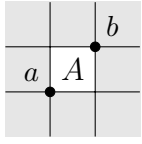

\begin{proof}
The only possible occurrence of $p$ can be formed by the element 1 in the leftmost position, and the largest element in the rightmost position. Any non-empty permutation  $\sigma'$, $\red(\sigma')\in K^s\cap K^\ell$, formed by the elements of $\sigma$ is counted by $C(t)-1$ given by (\ref{C(x)}). The king permutations avoiding $p$ are counted by $P(t)$. To get $P(t)$, we take all king permutations and subtract those that contain $p$:
$$P(t)=A(t)-t^2\big(C(t)-1\big), $$
where the factor of $t^2$ corresponds to $ab$, the only occurrence of $p$, presented schematically in Fig.~\ref{pic-triv-thm-13}. By Lemma~\ref{lem-length-1}, we obtain the formula for $P(t)$.

Similarly, we can get  $$E(t,u)=P(t)+ut^2\big(C(t)-1\big).$$ 
Indeed, any $\sigma \in K$ can contain at most one occurrence of $p$, and $\sigma \in K^s/K^{s\ell}$ avoids the pattern $p$. By Lemma~\ref{lem-length-1} , we get the formula for $E(t,u)$.
\end{proof}

\begin{thm}\label{triv-thm-19} 
Let $p = \pattern{scale = 0.5}{2}{1/1,2/2}{0/1,0/2,1/1,1/2,2/0,2/2}$, $E(t,u)=\sum_{n\geq 0}t^n\sum_{\sigma\in K_n}u^{p(\sigma)}$,
and $P(t)$ be the generating function for $K(p)$. Then, $$P(t)=\left(1+t-\frac{tA(t)}{1+t}\right)A(t),$$ $$E(t,u)=\left(1+t-ut-\frac{t(1-u)A(t)}{1+t}\right)A(t),$$
where $A(t)$ is given by (\ref{gf-kings}). The initial terms in the expansion of  $E(t,u)$ are
\begin{footnotesize}
$$1+t+2t^4+(12+2u)t^5+(76+14u)t^6+(556+90u)t^7+(4596+646u)t^8+\cdots.$$
\end{footnotesize}
\end{thm}

\begin{figure}
\begin{center}
	\begin{tikzpicture}[scale=0.7, baseline=(current bounding box.center)]
	\foreach \x/\y in {0/1,0/2,1/1,1/2,2/0,2/2}		    
		\fill[gray!20] (\x,\y) rectangle +(1,1);
	\draw (0.01,0.01) grid (2+0.99,2+0.99);
	\filldraw (1,1) circle (3pt) node[above left] {$a$};
	\filldraw (2,2) circle (3pt) node[above left] {$b$};
	\node  at (1.5,0.5) {$A$};
	\node  at (2.5,1.5) {$B$};
	\end{tikzpicture}
\end{center}
\caption{Related to the proof of Theorem~\ref{triv-thm-19}}\label{pic-triv-thm-19}
\end{figure}
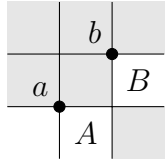

\begin{proof}
Suppose $\sigma \in K$ contains one occurrence of $p$, say $ab$, which is presented schematically in Fig.~\ref{pic-triv-thm-19}. There are two cases to consider here. 
\begin{itemize}
\item When $B$ is empty, then $A$ must be non-empty (or else, $b=a+1$ and $\sigma\not\in K$).  But then to the left of $b$ in $\sigma$ we must have a  non-empty permutation $\sigma' \in K^{\ell}$ (of length at least 2). Such king permutations are counted by $t\big(B(t)-1\big)$, where $t$ is contribution of $b$.

\item When $B$ is non-empty, then the permutation $\sigma''$ to the right of $b$ does not begin with the largest element, which is counted by $B(t)-1$. To the left of $b$ in $\sigma$ we must have a non-empty permutation $\sigma' \in K$, which is counted by $A(t)-1$. Hence, such king permutations are counted by $t\big(A(t)-1\big)\big(B(t)-1\big)$.
\end{itemize}

Now, to get avoidance, we take all king permutations and subtract those that contain the pattern $p$:
$$P(t)=A(t)-t\big(B(t)-1\big)-t\big(A(t)-1\big)\big(B(t)-1\big).$$
By Lemma~\ref{lem-length-1}, we get the formula for $P(t)$.

Similarly, we can get
\begin{equation}\label{formula-thm19-E}E(t,u)=P(t)+ut\big(B(t)-1\big)+ut\big(A(t)-1\big)\big(B(t)-1\big).\end{equation} 
Indeed, any $\sigma \in K$ can contain at most one occurrence of $p$. By solving \eqref{formula-thm19-E}, we get the formula for $E(t,u)$.
\end{proof}

\begin{thm}\label{triv-thm-20}
Let $p = \pattern{scale = 0.5}{2}{1/1,2/2}{0/0,0/1,0/2,1/1,1/2,2/0,2/1}$, $E(t,u)=\sum_{n\geq 0}t^n\sum_{\sigma\in K_n}u^{p(\sigma)}$, $F(t,u)=\sum_{n\geq 0}t^n\sum_{\sigma\in K_{n}^{s}}u^{p(\sigma)}$,
and $P(t)$ be the generating function for $K(p)$. Then, $$P(t)=\left(1+\frac{t^2}{1+t}-\frac{t^2A(t)}{(1+t)^2}\right)A(t),$$ $$E(t,u)=\left(1+\frac{t^2(1-u)}{1+t}-\frac{t^2(1-u)A(t)}{(1+t)^2} \right)A(t),$$
where $A(t)$ is given by (\ref{gf-kings}). The initial terms in the expansion of  $E(t,u)$ are
\begin{footnotesize}
$$1+t+2t^4+14t^5+(88+2u)t^6+(634+12u)t^7+(5164+78u)t^8+\cdots.$$
\end{footnotesize}
\end{thm}

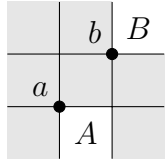
\begin{figure}
\begin{center}
	\begin{tikzpicture}[scale=0.7, baseline=(current bounding box.center)]
	\foreach \x/\y in {0/0,0/1,0/2,1/1,1/2,2/0,2/1}		    
		\fill[gray!20] (\x,\y) rectangle +(1,1);
	\draw (0.01,0.01) grid (2+0.99,2+0.99);
	\filldraw (1,1) circle (3pt) node[above left] {$a$};
	\filldraw (2,2) circle (3pt) node[above left] {$b$};
	\node  at (1.5,0.5) {$A$};
	\node  at (2.5,2.5) {$B$};
	\end{tikzpicture}
\end{center}
\caption{Related to the proof of Theorem~\ref{triv-thm-20}}\label{pic-triv-thm-20}
\end{figure}

\begin{proof}
Let us first consider king permutations containing one occurrence of $p$, say $ab$, depicted in Fig.~\ref{pic-triv-thm-20}. Clearly, $A$ is non-empty and the king permutation formed by $A$ together with $a$ is beginning with its largest element, which is counted by $A(t)-B(t)-t$. The permutation formed by $B$ together with $b$ is a permutation beginning with its smallest element, which is counted by $A(t)-B(t)$. 

For avoidance, we take all permutations and subtract those that contain $p$. Hence,
$$P(t)=A(t)-\big(A(t)-B(t)-t\big)\big(A(t)-B(t)\big).$$
By Lemma~\ref{lem-length-1}, we get the formula for $P(t)$.

Similarly, we can get
\begin{equation}\label{formula-thm20-E}E(t,u)=P(t)+u\big(A(t)-B(t)-t\big)\big(A(t)-B(t)\big).\end{equation} 
Indeed, any $\sigma \in K$ can contain at most one occurrence of $p$. By solving \eqref{formula-thm20-E},
we get the formula for $E(t,u)$.
\end{proof}

\begin{thm}\label{triv-thm-22} 
Let $p = \pattern{scale = 0.5}{2}{1/1,2/2}{0/1,1/2,0/0,2/0,2/2,1/1}$, $E(t,u)=\sum_{n\geq 0}t^n\sum_{\sigma\in K_n}u^{p(\sigma)}$,
and $P(t)$ be the generating function for $K(p)$. Then, $$P(t)=\left(1+t^2-\frac{t^2A^2(t)}{(1+t)^2}\right)A(t),$$ $$E(t,u)=\left(1+t^2(1-u)\big(1-\frac{A^2(t)}{(1+t)^2}\big)\right)A(t),$$
where $A(t)$ is given by (\ref{gf-kings}). The initial terms in the expansion of  $E(t,u)$ are
\begin{footnotesize}
$$1+t+2t^4+14t^5+(86+4u)t^6+(618+28u)t^7+(5062+180u)t^8+\cdots.$$
\end{footnotesize}
\end{thm}

\begin{figure}
\begin{center}
	\begin{tikzpicture}[scale=0.7, baseline=(current bounding box.center)]
	\foreach \x/\y in {0/1,1/2,0/0,2/0,2/2,1/1}		    
		\fill[gray!20] (\x,\y) rectangle +(1,1);
	\draw (0.01,0.01) grid (2+0.99,2+0.99);
	\filldraw (1,1) circle (3pt) node[above left] {$a$};
	\filldraw (2,2) circle (3pt) node[above left] {$b$};
	\node  at (0.5,2.5) {$A$};
	\node  at (1.5,0.5) {$B$};
	\node  at (2.5,1.5) {$C$};
	\end{tikzpicture}
\end{center}
\caption{Related to the proof of Theorem~\ref{triv-thm-22}}\label{pic-triv-thm-22}
\end{figure}
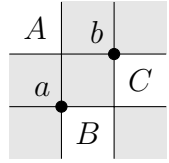

\begin{proof} 
Suppose that a permutation $\sigma\in K$ contains an occurrence $ab$ of the pattern $p$, that is presented schematically in Fig.~\ref{pic-triv-thm-22}.  Clearly,  $B$ and $C$ can not be empty at the same time, and $A$ is counted by $A(t)$. The following three cases are possible.
\begin{itemize}
\item Suppose that $B$ is non-empty and $C$ is empty. The permutation formed by $B$ together with $a$
is beginning with its largest element, which is counted by $A(t)-B(t)-t$. Hence, the king permutations in this case are counted by $t\big(A(t)-B(t)-t\big)A(t)$.

\item Suppose that $B$ is empty and $C$ is non-empty. The permutation formed by $C$ and $b$
begins with its largest element, which is counted by $A(t)-B(t)-t$. Thus, king permutations in this case are counted by $t\big(A(t)-B(t)-t\big)A(t)$.

\item Finally, suppose that $B$ is non-empty and $C$ is non-empty. Both the permutation formed by $B$ together with $a$ and the permutation formed by $C$ together with $b$ begin with their largest element and are of length at least 2. Therefore, king permutations in this case are counted by $\big(A(t)-B(t)-t\big)^2A(t)$.
\end{itemize}

As for avoidance, we clearly have
$$P(t)=A(t)-2t\big(A(t)-B(t)-t\big)A(t)-\big(A(t)-B(t)-t\big)^2A(t),$$
and by Lemma~\ref{lem-length-1}, we get the formula for $P(t)$.

Similarly, we can derive
\begin{equation}\label{formula-thm22-E}E(t,u)=P(t)+2ut\big(A(t)-B(t)-t\big)A(t)+u\big(A(t)-B(t)-t\big)^2A(t).\end{equation} 
Indeed, any $\sigma \in K$ can contain at most one occurrence of $p$. By solving \eqref{formula-thm22-E},
we get the formula for $E(t,u)$.
\end{proof}

\section{The other distributions}\label{other-distr-sec}

In this section, we derive the remaining distributions in this paper. 

\subsection{Distribution of the pattern Nr.\ 16}  

Our next theorem establishes the avoidance and  distribution of the pattern Nr.\ 16 = $\pattern{scale = 0.5}{2}{1/1,2/2}{0/1,2/0,1/0,0/2}$.

\begin{thm}\label{thm-pat-16} 
Let $p = \pattern{scale = 0.5}{2}{1/1,2/2}{0/1,2/0,1/0,0/2}$, $E(t,u)=\sum_{n\geq 0}t^n\sum_{\sigma\in K_n}u^{p(\sigma)}$, 
and $P(t)$ be the generating function for $K(p)$. Then,
\begin{footnotesize}
$$P(t)=\frac{(1+t)^2A(t)}{1+t+tA(t)},$$ 
and $E(t,u)$ is
\begin{equation}\label{thm-pat-16-relation-E}
E(t,u)=\sum_{i\geq 0}u^{i\choose 2}t^i(1+u^it)\prod_{j=0}^{i}\frac{(1+u^jt)A(u^jt)}{1+u^jt+u^jtA(u^jt)}\prod_{k=1}^{i}\frac{-1-u^kt+A(u^kt)}{(1+u^kt)A(u^kt)},
\end{equation}
\end{footnotesize}
where $A(t)$ is given by (\ref{gf-kings}). The initial terms in the expansion of  $E(t,u)$ are
\begin{footnotesize}
$$1+t+2t^4+(12+2u^4)t^5+(78+12u^5)t^6+(568+78u^6)t^7+(4674+568u^7)t^8+\cdots.$$
\end{footnotesize}
\end{thm}

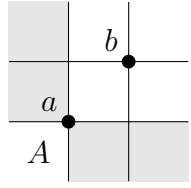
\begin{figure}
\begin{center}
	\begin{tikzpicture}[scale=0.8, baseline=(current bounding box.center)]
	\foreach \x/\y in {0/1,0/2,1/0,2/0}		    
		\fill[gray!20] (\x,\y) rectangle +(1,1);
	\draw (0.01,0.01) grid (2+0.99,2+0.99);
	\filldraw (1,1) circle (3pt) node[above left] {$a$};
	\filldraw (2,2) circle (3pt) node[above left] {$b$};
	\node  at (0.5,0.5) {$A$};
	\end{tikzpicture}
\end{center}
\caption{Related to the proof of Theorem~\ref{thm-pat-16}}\label{pic-thm-pat-16}
\end{figure}

\begin{proof} 
Suppose a king permutation contains at least one occurrence of $p$, say $ab$, which is presented schematically in Fig.~\ref{pic-thm-pat-16}. Assume that $a$ is the leftmost possible among all occurrences of $p$. Clearly,  $A$ must be \pattern{scale=0.5}{1}{1/1}{0/1,1/0}-avoiding, and the number of possibilities to choose $A$ is then given by $\frac{(1+t)A(t)}{1+t+tA(t)}$ by Theorem~\ref{thm-length-1}.
On the other hand, to the right of $a$ in $\sigma$ we can have any non-empty permutation $\sigma'$, $\red(\sigma')\in K^s$,  formed by the largest elements of $\sigma$ and counted by $B(t)-1$. Hence, such king permutations are counted by $t\big(B(t)-1\big)\frac{(1+t)A(t)}{1+t+tA(t)}$. By definition of $P(t)$, we have
$$P(t)=A(t)-t\big(B(t)-1\big)\frac{(1+t)A(t)}{1+t+tA(t)}.$$

Similarly, for $E(t,u)$, we can derive 
\begin{equation}\label{thm-pat-16-relation-E-1}
E(t,u)=P(t)+\big(E^*(t,u)-t\big)\frac{(1+t)A(t)}{1+t+tA(t)},
\end{equation}
where $E^*(t,u)$ records the distribution of $p$ on king permutations beginning with the smallest element. Indeed,  $E^*(t,u)$ is given by the element $a$ and all elements to the right of it, and to distinguish from the avoidance case (given by $P(t)$), to the right of $a$ we must have at least one element (resulting in ``$-t$'' corresponding to a single element $a$). A recursion for $E^*(t,u)$ is 
\begin{equation}\label{thm-pat-16-relation-E-star}
E^*(t,u)=t\big(E(ut,u)-E^*(ut,u)\big),
\end{equation}
because any king permutation beginning with element 1 has the property that each element to the right of 1, together with 1, will form an occurrence of~$p$ (explaining the term $E(ut,u)$), and it cannot have element 2 in the second position (explaining the term $-E^*(ut,u)$); the factor of $t$ corresponds to element 1. From (\ref{thm-pat-16-relation-E-1}), we have
\begin{equation}\label{E*(t,u)-eq}
E^*(t,u)=\frac{1+t+tA(t)}{(1+t)A(t)}E(t,u)-1.
\end{equation}
Using (\ref{E*(t,u)-eq}) in (\ref{thm-pat-16-relation-E-star}), we obtain
$$E(t,u)=\frac{(1+t)A(t)}{1+t+tA(t)}\left(1+t+t\frac{A(ut)-1-ut}{(1+ut)A(ut)}E(ut,u)\right),$$
and iterating this gives us (\ref{thm-pat-16-relation-E}).
\end{proof}

\subsection{Distribution of the pattern Nr.\ 17} 

Our next theorem establishes the avoidance and distribution of the pattern Nr.\ 17 = $\pattern{scale = 0.5}{2}{1/1,2/2}{0/1,1/2,0/0,2/0,1/0,0/2,2/1}$.

\begin{thm}\label{thm-pat-17}
Let $p=\pattern{scale = 0.5}{2}{1/1,2/2}{0/1,1/2,0/0,2/0,1/0,0/2,2/1}$, $E(t,u)=\sum_{n\geq 0}t^n\sum_{\sigma\in K_n}u^{p(\sigma)}$, 
and $P(t)$ be the generating function for $K(p)$. Then,
$$P(t)=\left(\frac{1}{1+t}+\frac{t(1+t)}{1+t+tA(t)}\right)A(t),$$ $$E(t,u)=\left(\frac{1}{1+t}+\frac{t(1+t)}{1+t\big(1+u+ut+(1-u)A(t)\big)} \right)A(t),$$
where $A(t)$ is given by (\ref{gf-kings}). The initial terms in the expansion of  $E(t,u)$ are
\begin{footnotesize}
$$1+t+2t^4+14t^5+(88+2u)t^6+(636+10u)t^7+(5174+68u)t^8+\cdots.$$
\end{footnotesize}
\end{thm}

\begin{proof} 
Clearly, any occurrence of $p$ must begin with the element 1 placed in the first position in a king permutation. If $\sigma \in K(p)$ does not begin with 1, then such permutations are counted by $B(t)$. The remaining avoidance case is that $\sigma \in K(p)$ begins with 1, and the permutation to the right of 1 avoids the pattern $\pattern{scale=0.5}{1}{1/1}{0/1,1/0}$. By Theorem~\ref{thm-length-1}, such permutations are counted by $t(\frac{(1+t)A(t)}{1+t+tA(t)})$. Hence,
$$P(t)=B(t)+\frac{t(1+t)A(t)}{1+t+tA(t)},$$
which gives the formula for $P(t).$

We next discuss $E(t,u)$. King permutations not beginning with the element 1 are counted by $B(t)$ and they do not have occurrences of $p$. Consider a king permutation $\sigma$ beginning with 1. Any occurrence of $\pattern{scale=0.5}{1}{1/1}{0/1,1/0}$ to the right of 1 leads to an occurrence of $p$ in $\sigma$. Hence, permutations in this case are counted by $tB(t,u)$, where the factor of $t$ corresponds to the element 1, and $B(t,u)$ is given in Theorem~\ref{thm-length-1}. Therefore, we have 
$$E(t,u)=B(t)+tB(t,u),$$
which completes our proof. 
\end{proof}

\subsection{Distribution of the pattern Nr.\ 27} 

Our next theorem establishes the avoidance and  distribution of the pattern Nr.\ 27 = $\pattern{scale=0.5}{2}{1/1,2/2}{0/1,2/0,2/2,1/0,1/1,0/2}$.

\begin{thm}\label{thm-pat-27}
Let $p=\pattern{scale=0.5}{2}{1/1,2/2}{0/1,2/0,2/2,1/0,1/1,0/2}$, $E(t,u)=\sum_{n\geq 0}t^n\sum_{\sigma\in K_n}u^{p(\sigma)}$, 
and $P(t)$ be the generating function for $K(p)$. Then,
\begin{footnotesize}
$$P(t)=\left(t+\frac{1}{1+t}-\frac{t^2A^2(t)}{(1+t)\big(1+t+tA(t)\big)}\right)A(t),$$ $$E(t,u)=\left(1+\frac{t^2(1-u)}{1+t}\left(1-\frac{A^2(t)}{1+t\big(1+u+ut+(1-u)A(t)\big)}\right)\right)A(t),$$
\end{footnotesize}
where $A(t)$ is given by (\ref{gf-kings}). The initial terms in the expansion of  $E(t,u)$ are
\begin{footnotesize}
$$1+t+2t^4+14t^5+(86+4u)t^6+(624+20u+2u^2)t^7+(5096+136u+10u^2)t^8+\cdots.$$
\end{footnotesize}  
\end{thm}

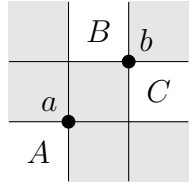
\begin{figure}
\begin{center}
	\begin{tikzpicture}[scale=0.8, baseline=(current bounding box.center)]
	\foreach \x/\y in {0/1,0/2,1/0,1/1,2/0,2/2}		    
	\fill[gray!20] (\x,\y) rectangle +(1,1);
	\draw (0.01,0.01) grid (2+0.99,2+0.99);
	\filldraw (1,1) circle (3pt) node[above left] {$a$};
	\filldraw (2,2) circle (3pt) node[above right] {$b$};
	\node  at (0.5,0.5) {$A$};
	\node  at (1.5,2.5) {$B$};
	\node  at (2.5,1.5) {$C$};
	\end{tikzpicture}
\end{center}
\caption{Related to the proof of Theorem~\ref{thm-pat-27}}\label{av-pic-pat-27}
\end{figure}

\begin{proof}
Suppose that a permutation $\sigma\in K$ contains at least one occurrence of the pattern $p$, and $ab$ in Fig.~\ref{av-pic-pat-27} is an occurrence with the leftmost possible $a$. $B$ and $C$ cannot be empty at the same time, or else $\sigma\not\in K$. 

We know that $A\in K^{\ell}$, which is counted by $B(t)$. We assume that $b$ is the leftmost possible after $a$ is fixed as the leftmost possible, and so $B\in K(\pattern{scale=0.5}{1}{1/1}{0/0,1/1})$ (implying $B$ never ends with the smallest element) and enumeration of possible $B$ is given by (\ref{avd-begin-smallest-dist}), namely, by  $\frac{(1+t)A(t)}{1+t+tA(t)}$. The permutation formed by $C$ does not begin with the largest element, which is counted by $B(t)$. Therefore, the number of permutations in $K(p)$ is
$$P(t)=A(t)-\left(t^2B^2(t)\frac{(1+t)A(t)}{1+t+tA(t)}-t^2B(t)\right),$$
where we subtracted $t^2B(t)$ corresponding to the case of $B$ and $C$ being empty at the same time. This gives the claimed formula for $P(t)$ using Lemma~\ref{lem-length-1}.

As for distribution $E(t,u)$, permutations in $K(p)$ are counted by $P(t)$. On the other hand, if  $\sigma\in K$ contains at least one occurrence of $p$, then using our considerations above, such permutations are enumerated by
\begin{equation}\label{aux-formula-1}
ut^2\cdot \frac{A(t)}{1+t} \cdot \frac{(1+t)A(t)}{1+t+tA(t)} \cdot B(t,u)-ut^2B(t)
\end{equation}
where $A(t)$ and $B(t,u)$ are given by (\ref{gf-kings}) and (\ref{not-begin-end-extreme}), respectively. Indeed, $ab$ contributes the factor of $ut^2$ and there are $\frac{A(t)}{1+t}$ choices for $A$. The number of choices for $B$ is $\frac{(1+t)A(t)}{1+t+tA(t)}$ because $B$ avoids $\pattern{scale=0.5}{1}{1/1}{0/0,1/1}$ and hence never ends with the smallest element. Finally, the choices for $C$ are clearly given by $B(t,u)$, and we need to subtract the cases when both $B$ and $C$ are empty. 

Using the formula (\ref{not-begin-end-extreme}) for $B(t,u)$ in (\ref{aux-formula-1}), we obtain the claimed formula for $E(t,u)$.
\end{proof}

\subsection{Distribution of the pattern Nr.\ 28}

Our next theorem establishes the avoidance and  distribution of the pattern  
Nr.\ 28 = $\pattern{scale=0.5}{2}{1/1,2/2}{0/1,1/2,0/0,2/2,1/0,2/1}$.

\begin{thm}\label{thm-pat-28}
Let $p=\pattern{scale=0.5}{2}{1/1,2/2}{0/1,1/2,0/0,2/2,1/0,2/1}$, $E(t,u)=\sum_{n\geq 0}t^n\sum_{\sigma\in K_n}u^{p(\sigma)}$, 
and $P(t)$ be the generating function for $K(p)$. Then,
$$P(t)=\frac{(1+t)^2A(t)}{(1+t)^2+t^2\big(A(t)-t-1\big)A(t)},$$ $$E(t,u)=\frac{(1+t)^2A(t)}{1+t\big(2+t(1+(1-u)(A(t)-t-1)A(t))\big)},$$
where $A(t)$ is given by (\ref{gf-kings}). The initial terms in the expansion of  $E(t,u)$ are
\begin{footnotesize}
$$1+t+2t^4+14t^5+(88+2u)t^6+(632+14u)t^7+(5152+90u)t^8+\cdots.$$
\end{footnotesize}
\end{thm}

\begin{figure}[!ht]
\begin{center}
	\begin{tikzpicture}[scale=0.8, baseline=(current bounding box.center)]
	\foreach \x/\y in {0/0,0/1,1/0,1/2,2/1,2/2}		    
	\fill[gray!20] (\x,\y) rectangle +(1,1);
	\draw (0.01,0.01) grid (2+0.99,2+0.99);
	\filldraw (1,1) circle (3pt) node[above left] {$a$};
	\filldraw (2,2) circle (3pt) node[above left] {$b$};
	\node  at (0.5,2.5) {$A$};
	\node  at (1.5,1.5) {$B$};
	\node  at (2.5,0.5) {$C$};
	\end{tikzpicture}
	\caption{Related to the proof of Theorem~\ref{thm-pat-28}}\label{pic-thm-pat-28}
\end{center}
\end{figure}
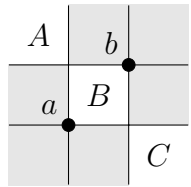

\begin{proof}
Suppose that a permutation $\sigma\in K$ contains at least one occurrence of the pattern $p$, and $ab$ in Fig.~\ref{pic-thm-pat-28} is an occurrence with the leftmost possible $ab$. Hence, $A$ avoids the pattern $p$, which is counted by $P(t)$ (there are no extra restrictions on $A$). The permutation formed by $B$ can be any non-empty permutation in $K^{s\ell}$ counted by $C(t)-1$ given by Theorem~\ref{triv-thm-13}. Also, $C$ can be any king permutation, which is counted by $A(t)$. Hence,
$$P(t)=A(t)-t^2P(t)A(t)\big(C(t)-1\big)$$
giving us the formula for $P(t)$.

Similarly, we have
$$E(t,u)=P(t)+ut^2P(t)\big(C(t)-1\big)E(t,u),$$
where $C$ contributes $E(t,u)$ because the occurrences of $p$ in $C$ lead to the occurrences of
$p$ in $\sigma\in K$. This gives the formula for $E(t,u)$ and completes our proof.
\end{proof}

\subsection{Distribution of the pattern Nr.\ 33} 

Our next theorem establishes the avoidance and the distribution of the pattern  
Nr.\ 33 = $\pattern{scale=0.5}{2}{1/1,2/2}{0/1,1/2,2/0,1/0,0/2,2/1}$.

\begin{thm}\label{thm-pat-33}
Let $p=\pattern{scale=0.5}{2}{1/1,2/2}{0/1,1/2,2/0,1/0,0/2,2/1}$, $E(t,u)=\sum_{n\geq 0}t^n\sum_{\sigma\in K_n}u^{p(\sigma)}$, 
and $P(t)$ be the generating function for $K(p)$. Then,
\begin{footnotesize}
$$P(t)=\frac{(1+t)\big(1+t+t(2+t)A(t)\big)A(t)}{\big(1+t+tA(t)\big)^2},$$ $$E(t,u)=\frac{(1+t)\big(1+t(1+u+ut+(2-u+t)A(t))\big)A(t)}{\big(1+t+tA(t)\big)\big(1+t(1+u+ut+(1-u)A(t))\big)},$$
\end{footnotesize}
where $A(t)$ is given by (\ref{gf-kings}). The initial terms in the expansion of  $E(t,u)$ are
\begin{footnotesize}
$$1+t+2t^4+14t^5+(88+2u)t^6+(636+10u)t^7+(5174+68u)t^8+\cdots.$$
\end{footnotesize}
\end{thm}

\begin{figure}
\begin{center}
	\begin{tikzpicture}[scale=0.8, baseline=(current bounding box.center)]
	\foreach \x/\y in {0/1,1/2,2/0,1/0,0/2,2/1}		    
	\fill[gray!20] (\x,\y) rectangle +(1,1);
	\draw (0.01,0.01) grid (2+0.99,2+0.99);
	\filldraw (1,1) circle (3pt) node[above left] {$a$};
	\filldraw (2,2) circle (3pt) node[above left] {$b$};
	\node  at (0.5,0.5) {$A$};
	\node  at (1.5,1.5) {$B$};
	\node  at (2.5,2.5) {$C$};
	\end{tikzpicture}
\end{center}
\caption{Related to the proof of  Theorem~\ref{thm-pat-33}}\label{pic-thm-pat-33}
\end{figure}
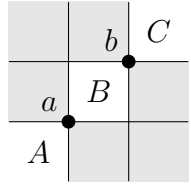

\begin{proof}
Suppose that a permutation $\sigma\in K$ contains at least one occurrence of $p$, and $ab$ in Fig.~\ref{pic-thm-pat-33} is an occurrence with the leftmost possible $ab$, which contributes the factor of $ut^2$. The king permutation $\sigma'\in K^\ell$ in $A$ avoids the pattern $\pattern{scale=0.5}{1}{1/1}{0/1,1/0}$, which is given by (\ref{avd-begin-smallest-dist}), namely, by  $\frac{(1+t)A(t)}{1+t+tA(t)}$. 
The permutation formed by $B$ can be any non-empty permutation in $K^{s\ell}$ that avoids  $\pattern{scale=0.5}{1}{1/1}{0/1,1/0}$ (since $b$ is the leftmost possible). These permutations are counted by $C(t,0)-1$ given by Theorem~\ref{triv-thm-13}. Finally, $C$ can be any permutation in $K^{s}$, which is counted by $B(t)$.
Therefore, we obtain
$$P(t)=A(t)-t^2\frac{(1+t)A(t)}{1+t+tA(t)}B(t)\big(C(t,0)-1\big).$$

Similarly, we can derive
$$E(t,u)=P(t)+ut^2\frac{(1+t)A(t)}{1+t+tA(t)}B(t,u)\big(C(t,0)-1\big),$$
where $C$ contributes $B(t,u)$ because the occurrences of $p$ in $C$ lead to the occurrences of
$p$ in $\sigma\in K$. Also, because both $A$ and $B$ are \pattern{scale=0.5}{1}{1/1}{0/1,1/0}-avoiding, no extra occurrence of $p$ can be introduced. Using the formula for $B(t,u)$ in (\ref{aux-formula-1}) we obtain $E(t,u)$.
\end{proof}

\subsection{Distribution of the pattern Nr.\ 55}

Our next theorem establishes the avoidance and  distribution of the pattern  
Nr.\ 55 = $\pattern{scale = 0.5}{2}{1/1,2/2}{0/1,1/2,0/0,2/0,1/1,2/1}$.

\begin{thm}\label{thm-pat-55}
Let $\pattern{scale = 0.5}{2}{1/1,2/2}{0/1,1/2,0/0,2/0,1/1,2/1}$, $E(t,u)=\sum_{n\geq 0}t^n\sum_{\sigma\in K_n}u^{p(\sigma)}$, 
and $P(t)$ be the generating function for $K(p)$. Then,
$$P(t)=\frac{(1+t)(A(t)-t)}{1+t(A(t)-t-1)},$$
$$E(t,u)=\frac{(1+t)\big(t(1-u)-(1-ut)A(t)\big)}{-1+t\big(1-u+t-(1-u)A(t)\big)},$$
where $A(t)$ is given by (\ref{gf-kings}). The initial terms in the expansion of  $E(t,u)$ are
\begin{footnotesize}
$$1+t+2t^4+14t^5+(88+2u)t^6+(632+14u)t^7+(5152+88u+2u^2)t^8+\cdots.$$
\end{footnotesize}
\end{thm}

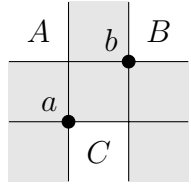
\begin{figure}[!ht]
\begin{center}
	\begin{tikzpicture}[scale=0.8, baseline=(current bounding box.center)]
	\foreach \x/\y in {0/0,0/1,1/1,1/2,2/0,2/1}		    
	\fill[gray!20] (\x,\y) rectangle +(1,1);
	\draw (0.01,0.01) grid (2+0.99,2+0.99);
	\filldraw (1,1) circle (3pt) node[above left] {$a$};
	\filldraw (2,2) circle (3pt) node[above left] {$b$};
	\node  at (0.5,2.5) {$A$};
	\node  at (2.5,2.5) {$B$};
	\node  at (1.5,0.5) {$C$};
	\end{tikzpicture}
		\caption{Related to the proof of Theorem~\ref{thm-pat-55}}\label{pic-thm-pat-55}
\end{center}
\end{figure}

\begin{proof}
We claim that 
\begin{equation}\label{av-pattern-55} 
P(t)+(B(t)-1)(P(t)-1)(t+t^2)=A(t). 
\end{equation}
Indeed, each permutation $\sigma\in K$, counted by $A(t)$,  either avoids $p$ (which is counted by the $P(t)$ term in \eqref{av-pattern-55}), or  contains at least one occurrence of $p$. Among all such occurrences, pick an occurrence $ab$ where $a$ is  the \emph{highest} possible,  so that, referring to Fig.~\ref{pic-thm-pat-55}, the sub-permutation given by $AbB$ is $p$-avoiding, and $C$ is any non-empty king permutation not beginning with the largest element. The possibilities for $C$ are counted by  $B(t)-1$. The possibilities for $AbB$ are given by $P(t)-1$. It remains to explain the factor of $(t+t^2)$ in (\ref{av-pattern-55}). In this factor, $t$ corresponds to $a$ in the situation when $AbB$ is a king permutation. However, instead of a king permutation we can begin with a non-king permutation $AbB$ where the only violation will be the situation that $A$ ends with element $b+1$. This violation will disappear once the element $a$ is inserted. Solving (\ref{av-pattern-55}) for $P(t)$, and using the formula (\ref{B(x)}) for $B(t)$, we obtain the claimed formula for $P(t)$.	

Similarly to the avoidance case, we can derive a functional equation for $E(t,u)$, which is
\begin{equation}\label{eq-E-pattern-55}E(t,u)=P(t)+u(E^*(t,u)-1)(P(t)-1)(t+t^2),\end{equation}
where $E^*(t,u)$ is the distribution of $p$ on king permutations not beginning with the largest element and the factor of $u$ corresponds to the occurrence $ab$. Here we used the fact that when merging together $AbB$, $a$ and $C$, no occurrence of $p$ will disappear, and all occurrence of $p$ in $\sigma$ are those in $C$ and the occurrence $ab$. The formula for $E(t,u)$ can now be obtained from (\ref{eq-E-pattern-55}) after observing the relation $E^*(t,u)=E(t,u)-tE^*(t,u)$.
\end{proof}

\subsection{Distribution of the pattern Nr.\ 63} 

Our next theorem establishes the avoidance and  distribution of the pattern  
Nr.\ 63 = $\pattern{scale = 0.6}{2}{1/1,2/2}{0/1,1/2,0/0,2/1,2/0}$.

\begin{thm}\label{thm-pat-63}
Let $p=\pattern{scale = 0.6}{2}{1/1,2/2}{0/1,1/2,0/0,2/1,2/0}$, $E(t,u)=\sum_{n\geq 0}t^n\sum_{\sigma\in K_n}u^{p(\sigma)}$, 
and $P(t)$ be the generating function for $K(p)$. Then,
\begin{footnotesize}
$$P(t)=\frac{2A(t)-t-1}{A(t)-t},$$ 
$$E(t,u)=\frac{(1+t)(1-u)+(-2+u+ut^2)A(t)}{-u+t(1-u+ut)-(1-u)A(t)},$$
\end{footnotesize}
where $A(t)$ is given by (\ref{gf-kings}). The initial terms in the expansion of  $E(t,u)$ are
\begin{footnotesize}
$$1+t+2t^4+(12+2u)t^5+(76+14u)t^6+(556+88u+2u^2)t^7+(4592+636u+14u^2)t^8+\cdots.$$
\end{footnotesize}
\end{thm}

\begin{figure}[!ht]
\begin{center}
	\begin{tikzpicture}[scale=0.8, baseline=(current bounding box.center)]
	\foreach \x/\y in {0/1,1/2,0/0,2/1,2/0}
	\fill[gray!20] (\x,\y) rectangle +(1,1);
	\draw (0.01,0.01) grid (2+0.99,2+0.99);
	\filldraw (1,1) circle (3pt) node[above left] {$a$};
	\filldraw (2,2) circle (3pt) node[above left] {$b$};
	\node  at (0.5,2.5) {$A$};
	\node  at (2.5,2.5) {$B$};
	\node  at (1.5,1.5) {$C$};
	\node  at (1.5,0.5) {$D$};
	\end{tikzpicture}
		\caption{Related to the proof of Theorem~\ref{thm-pat-63}}\label{pic-thm-pat-63}
\end{center}
\end{figure}
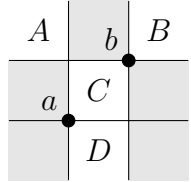

\begin{proof}
We claim that 
\begin{equation}\label{av-pattern-63} 
P(t)+(P(t)-1)(B(t)-1)(1+t)=A(t). 
\end{equation}
Indeed, each permutation $\sigma\in K$, which is counted by the right-hand side in \eqref{av-pattern-63},  either avoids $p$ (which is counted by the $P(t)$ term in \eqref{av-pattern-63}), or  contains at least one occurrence of $p$. Among all such occurrences, pick the occurrence $ab$ with the \emph{highest} possible $b$ as shown in Fig.~\ref{pic-thm-pat-63}. Referring to this figure, we note that the permutation formed by $A$ and $B$ together with $b$ must be nonempty and $p$-avoiding. There are two possible cases. 
\begin{itemize}
\item $\red(AbB)$ is a $p$-avoiding king permutation counted by $P(t)-1$ in \eqref{av-pattern-63}. Since $ab$ is an occurrence of $p$, in the part below $b$, which is formed by $C$ and $D$ together with $a$, we can have any non-empty king permutation that cannot end with the largest element, and hence is counted by $B(t)-1$. 

\item $\red(AbB)$ is a $p$-avoiding almost king permutation, with the only violation that $A$ ends with $b+1$ (the smallest element in $A$), which is  counted by $t(P(t)-1)$, because we can take any non-empty king permutation and replace its minimum element 1 by 21 and increase all other elements by 1 (such a replacement does not introduce an occurrence of $p$ in $AbB$); the factor of $t$ is given by the element 2. The rest of $\sigma$ is again given by $B(t)-1$. 
\end{itemize}

The considerations above give us (\ref{av-pattern-63}), and using the formula (\ref{B(x)}) for $B(t)$, we get the formula for $P(t)$.

For the distribution, we have the following functional equation:
\begin{equation}\label{dis-pattern-63} 
E(t,u)=P(t)+\frac{(E^*(t,u)-t)(P(t)-1)}{t}+(E^*(t,u)-t)(P(t)-1), 
\end{equation}
where $E^*(t,u)$ is the generating function for the distribution of $p$ on king permutations ending with the largest element. Indeed, either a permutation $\sigma\in K$ avoids $p$, which is given by $P(t)$, or it contains an occurrence $ab$ of $p$, and as above, we can let $b$ be the highest possible. We think of assembling $\sigma$ (with an occurrence of $p$) from two permutations, $\sigma'=AbB$ and the permutation $\sigma''$ formed by the elements in $\{1,2,\ldots,b\}$ in $\sigma$, so that $\sigma''$ is a king permutation ending with the largest element, and we will be glueing the smallest element in $\sigma'$ ($b$) with the largest element in $\sigma''$ ($b$). Similarly to the case of avoidance we have two cases to consider corresponding to the bullet points above. The first (resp., second) bullet point will give the second (resp., third) term in the right-hand side in (\ref{dis-pattern-63}), where division by $t$ corresponds to gluing two elements. Hence, (\ref{dis-pattern-63}) is explained. 

We next note that $E^*(t,u)$ satisfies
\begin{equation}\label{E'-63} 
E^*(t,u)=t+ut(E(t,u)-1-E^*(t,u)).
\end{equation}
Indeed, either a permutation counted by $E^*(t,u)$ is of length one (giving the term of $t$), or the leftmost element and the rightmost element form an occurrence of $p$, giving the factor $ut$, where $t$ corresponds to the largest element. Note that the largest element cannot be involved in any other occurrence of $p$. Finally, note that the elements to the left of the largest (rightmost) element form a non-empty king permutation not ending with the largest element.

Substituting (\ref{E'-63}) in (\ref{dis-pattern-63}), and using the formula for $P(t)$, we get the formula for $E(t,u)$.
\end{proof}

\subsection{Distribution of the pattern Nr.\ 64}

\begin{thm}\label{thm-pat-64}
Let $p=\pattern{scale = 0.5}{2}{1/1,2/2}{0/1,1/2,2/0,0/2,1/1}$, $E(t,u)=\sum_{n\geq 0}t^n\sum_{\sigma\in K_n}u^{p(\sigma)}$, 
and $P(t)$ be the generating function for $K(p)$. Then,
\begin{footnotesize}
$$P(t)=1+t+\frac{1}{1+t}-\frac{1}{A(t)},$$ 
$$E(t,u)=\frac{(-1+u)(1+ut)(1+t)+\big(2-u+t(2+t+u-u^2)\big)A(t)}{u(1+ut)(1+t)+(1-u)(1+t+ut)A(t)},$$
\end{footnotesize}
where $A(t)$ is given by (\ref{gf-kings}). The initial terms in the expansion of  $E(t,u)$ are
\begin{footnotesize}
$$1+t+2t^4+(10+4u)t^5+(68+20u+2u^2)t^6+(500+136u+10u^2)t^7+(4170+1004u+68u^2)t^8+\cdots.$$
\end{footnotesize}
\end{thm}

\begin{figure}[!ht]
\begin{center}
	\begin{tikzpicture}[scale=0.8, baseline=(current bounding box.center)]
	\foreach \x/\y in {0/1,1/2,2/0,0/2,1/1}
	\fill[gray!20] (\x,\y) rectangle +(1,1);
	\draw (0.01,0.01) grid (2+0.99,2+0.99);
	\filldraw (1,1) circle (3pt) node[above left] {$a$};
	\filldraw (2,2) circle (3pt) node[above left] {$b$};
	\node  at (2.5,2.5) {$C$};
	\node  at (2.5,1.5) {$D$};
	\node  at (0.5,0.5) {$A$};
	\node  at (1.5,0.5) {$B$};
	\end{tikzpicture}
		\caption{Related to the proof of Theorem~\ref{thm-pat-64}}\label{av-pic-pat-64}
\end{center}
\end{figure}
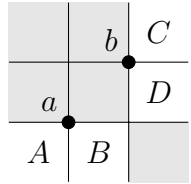

\begin{proof}
Suppose that a permutation $\sigma\in K$ contains at least one occurrence of $p$, and $ab$ in Fig.~\ref{av-pic-pat-64} is an occurrence with the leftmost possible $a$. Note that $B$ and $D$ cannot be empty at the same time, or else $\sigma\not\in K$. 

The non-empty permutation formed by $A$, $B$, together with $a$ avoids $p$, which is counted by $P(t)-1$. The non-empty permutation formed by $C$, $D$, together with $b$ is counted by $A(t)-1$. Therefore, the number of permutations in $K(p)$ is
$$P(t)=A(t)-\big((P(t)-1)(A(t)-1)-t^2B(t)\big),$$
where we subtracted $t^2B(t)$ corresponding to the case of $B$ and $D$ being empty at the same time (resulting in a non-king permutation). For the subtracted term,  $C$ gives $B(t)$ choices and $A$ must be empty, or else its maximum element together with $a$ would form an occurrence of $p$ contradicting $ab$ being the leftmost occurrence. This gives the claimed formula for $P(t)$ using Lemma~\ref{lem-length-1}. 

As for distribution $E(t,u)$, permutations in $K(p)$ are counted by $P(t)$. On the other hand, if  $\sigma\in K$ contains at least one occurrence of $p$, then using our considerations above, such permutations are enumerated by
$$u\big(P(t)-1\big)\big(E(t,u)-1\big)-utE^*(t,u),$$
where $E^*(t,u)$ records the distribution of $p$ on king permutations beginning with the smallest element. Indeed, the non-empty permutation formed by $C$, $D$ together with $b$ is counted by $E(t,u)-1$ as no occurrence of $p$ can start to the left of $b$ and end to the right of $b$ (due to the element $b$). However, we need to subtract the cases when both $B$ and $D$ are empty.

Considering the case when both $B$ and $D$ are empty, $E^*(t,u)$ is given by the element $b$ and $C$. Hence,  
$$E^*(t,u)=t+ut\big(E(t,u)-E^*(t,u)-1\big).$$
Indeed, $A$ must be empty, or else $ab$ in Fig.~\ref{av-pic-pat-64} would not be the leftmost occurrence of $p$. The permutation formed by $C$ together with $b$ can cause occurrences of $p$ (explaining the term $utE(t,u)$), and it cannot have element $b+1$ in the second position (explaining the term $-utE^*(t,u)$). The initial factor of $t$ corresponds to the case when $C$ is empty. We obtain the claimed formula for $E(t,u)$.
\end{proof}
 
\section{Concluding remarks}\label{final-sec}

In this paper, we provide  distribution of 22 mesh patterns  of length 2, appearing in \cite{KZ}, on king permutations. However, there are  other patterns discussed in  \cite{KZ}, and additionally,  pattern $\pattern{scale = 0.5}{2}{1/1,2/2}{0/0,0/1,0/2,1/0,1/1,2/0}$ (Nr.\ 66) from \cite{KZZ}, for which we were unable to determine the distributions on king permutations. These patterns can be found in Table~\ref{tab-2}, and we invite the reader to find distribution of these patterns on king permutations.   

  \begin{table}[!ht]
 	{
 		\renewcommand{\arraystretch}{1.3}
\begin{footnotesize}
 \begin{center} 
 		\begin{tabular}{|c|c||c|c|}
 			\hline
 			{\footnotesize Nr.\ } & {\footnotesize Repr.\ $p$}
 			&  {\footnotesize Nr.\ } & {\footnotesize Repr.\ $p$}   
 			\\[5pt]
 			\hline		\hline
 		3 &  $\pattern{scale = 0.5}{2}{1/1,2/2}{0/0,0/1,1/2}$ 
 & 
  		18 & $\pattern{scale = 0.5}{2}{1/1,2/2}{0/0,0/1,0/2,1/2,2/0,2/2}$   
\\[5pt]
 			\hline
 		5 & $\pattern{scale = 0.5}{2}{1/1,2/2}{0/0,0/1,0/2}$ 
 		&
		21 & $\pattern{scale = 0.5}{2}{1/1,2/2}{0/1,1/2,0/0,2/0,2/2}$   
 			\\[5pt] 
 			\hline
 		8 &  $\pattern{scale = 0.5}{2}{1/1,2/2}{0/0,0/1,1/0,1/1}$
 		& 
 		56 & $\pattern{scale = 0.5}{2}{1/1,2/2}{0/1,1/2,0/0,2/2,1/1,2/1}$
 		    \\[5pt] 
 			\hline
 		9 & $\pattern{scale = 0.5}{2}{1/1,2/2}{0/1,1/1,1/2,2/1}$
 			&
		65 & $\pattern{scale = 0.5}{2}{1/1,2/2}{0/1,1/0,0/0,1/1,2/2}$     
 			\\[5pt]
 			\hline
 		15 & $\pattern{scale = 0.5}{2}{1/1,2/2}{0/1,0/2,1/0,1/1,1/2}$
 			&
		66 & $\pattern{scale = 0.5}{2}{1/1,2/2}{0/0,0/1,0/2,1/0,1/1,2/0}$    
 			\\[5pt]
 			\hline
 		\end{tabular}
\end{center} 	
\end{footnotesize}}
 	\caption{The patterns whose distribution is known on all permutations \cite{KZ,KZZ}, yet remain unknown for king permutations. Patterns' numbers come from \cite{Hilmarsson2015Wilf,KZ,KZZ}.}\label{tab-2}
\end{table}

\vspace{0.5cm}

\noindent
{\bf Acknowledgement.} The authors are grateful to Sergey Kitaev for his useful discussions related to this paper. Zhang was supported by the National Natural Science Foundation of China (No.\ 12171362). \\

\end{document}